\documentclass[a4paper]{article}

\usepackage{tikz,pgfpages}

\newcommand{\bx}{\textbf{x}}
\newcommand{\bF}[1]{\textbf{#1}}

\newcommand{\norm}[2]{\lVert #1 \rVert_{#2}}
\newcommand{\normtwo}[1]{\lVert #1 \rVert_2}
\newcommand{\define}{\stackrel{\text{def}}{=}}
\newcommand{\at}[2]{\left.#1\right|_{#2}}
\newcommand{\Span}[1]{\text{span}\left\{#1\right\}}
\newcommand{\krylov}[2]{{\cal{K}}_{#2}(#1)}
\usepackage{algorithm}
\usepackage{algorithmic}
\usepackage{epstopdf}
\usepackage{graphicx}
\usepackage{amsmath,amssymb,amsthm}

\newtheorem{propos}{Proposition}

\title{A Flexible Krylov Solver for Shifted Systems with Application to Oscillatory Hydraulic Tomography}

\author{Arvind K. Saibaba \and Tania Bakhos \and Peter K. Kitanidis}

\begin{document}

\maketitle

\begin{abstract}
We discuss efficient solutions to systems of shifted linear systems arising in computations for oscillatory hydraulic tomography (OHT). The reconstruction of hydrogeological parameters such as hydraulic conductivity and specific storage using limited discrete measurements of pressure (head) obtained from sequential oscillatory pumping tests, leads to a nonlinear inverse problem. We tackle this using the quasi-linear geostatistical approach~\cite{kitanidis1995quasi}. This method requires repeated solution of the forward (and adjoint) problem for multiple frequencies, for which we use flexible preconditioned Krylov subspace solvers specifically designed for shifted systems based on ideas in~\cite{gu2007flexible}. The solvers allow the preconditioner to change at each iteration. We analyze the convergence of the solver and perform an error analysis when an iterative solver is used for inverting the preconditioner matrices. Finally, we apply our algorithm to a challenging application taken from oscillatory 
hydraulic tomography to demonstrate the computational gains by using the resulting method.

\end{abstract}

\section{Introduction}

Hydraulic tomography (HT) is a method for characterizing the subsurface that consists of applying pumping in wells while aquifer pressure (head) responses are measured. Using the data collected at various locations, important aquifer parameters (e.g., hydraulic conductivity and specific storage) are estimated. An example of such a technique is transient hydraulic tomography (reviewed in~\cite{cardiff20113D}). 
Oscillatory hydraulic tomography (OHT) is an emerging technology for aquifer characterization that involves a tomographic analysis of oscillatory signals. Here we consider that a sinusoidal signal of known frequency is imposed at an injection point and the resulting change in pressure is measured at receiver wells. Consequently, these measurements are processed using a 
nonlinear inversion algorithm to recover estimates for the desired aquifer parameters. Oscillatory hydraulic tomography has notable advantages over transient hydraulic tomography; namely, a weak signal can be distinguished from the ambient noise and by using signals of different frequencies, we are able to extract additional information without having to drill additional wells.

Using multiple frequencies for OHT has the potential to improve the quality of the image. However, it involves considerable computational burden. Solving the inverse problem, i.e. reconstructing the hydraulic conductivity field from pressure measurements, requires several application of the forward (and adjoint) problem for multiple frequencies. As we shall show in section~\ref{sec:application}, solving the forward (and adjoint) problem involves the solution of shifted systems for multiple frequencies. For finely discretized grids, the cost of solving the system of equations corresponding to each frequency can be high to the extent that it might prove to be computationally prohibitive when many frequencies are used, for example, on the order of $200$. The objective is to develop an approach in which the cost of solving the forward (and adjoint) problem for multiple frequencies is not significantly higher than the cost of solving the system of equations for a single frequency - in other words, the cost should 
depend only weakly on the number of frequencies.

Direct methods, such as sparse LU, Cholesky or LDL$^T$ factorization, are suited to linear systems in which the matrix bandwidth is small, so that the fill-in is somewhat limited. An additional difficulty that direct methods pose is that for solving a sequence of shifted systems, the matrix has to be re-factorized for each frequency, resulting in a considerable computational cost. By contrast, Krylov subspace methods for shifted systems are particularly appealing since they exploit the shift-invariant property of Krylov subspaces~\cite{simoncini2007recent} to obtain approximate solutions for all frequencies by generating a single approximation space that is shift independent. Several algorithms have been developed for dealing with shifted systems. Some are based on Lanczos recurrences for symmetric systems~\cite{meerbergen2003solution,meerbergen2010lanczos}; others use the unsymmetric Lanczos~\cite{freund1993solution},  and some others use Arnoldi iteration~\cite{datta1991arnoldi,frommer1998restarted,
simoncini2003restarted,darnell2008deflated,gu2007flexible}. Shifted systems also occur in several other applications such as control theory, time dependent partial differential equations, structural dynamics, and quantum chromodynamics (see~\cite{simoncini2003restarted} and references therein). Hence, several other communities can benefit from advances in efficient solvers for shifted systems. The Krylov subspace method that we propose is closest in spirit to~\cite{gu2007flexible}. However, as we shall demonstrate, we have extended their solver significantly.

\textbf{Contributions}: Our major contributions can be summarized as follows:
\begin{itemize}
\item We have extended the flexible Arnoldi algorithm discussed in~\cite{gu2007flexible} for shifted systems of the form $(A + \sigma_jI)x_j =b$ to systems of the form $(K + \sigma_j M) x_j = b $ for $j = 1,\dots,n_f$ that employs multiple preconditioners of the form $(K+\tau M)$. In addition, we provide some analysis for  the convergence of the solver.

\item When an iterative solver is used to apply the preconditioner, we derive an error analysis that gives us stopping tolerances for monitoring convergence without constructing the full residual.

\item Our motivation for the need for fast solvers for shifted systems comes from oscillatory hydraulic tomography. We describe the key steps involved in inversion for oscillatory hydraulic tomography, and discuss how the computation of the Jacobian can be accelerated by the use of the aforementioned fast solvers.
\end{itemize}

\textbf{Limitations}: The focus of this work has been on the computational aspects of oscillatory hydraulic tomography. Although the initial results are promising, several issues remain to be resolved for application to realistic problems of oscillatory hydraulic tomography. For example, we are inverting for the hydraulic conductivity assuming that the storage field is known. In practice, the storage is also unknown and needs to be estimated from the data as well. Moreover, simulating realistic conditions (higher variance in the log conductivity field, and adding measurement noise in a realistic manner) may significantly improve the performance with the addition of information from different frequencies. We will deal with these issues in another paper.

The paper is organized as follows. In section~\ref{sec:krylov}, we discuss the Krylov subspace methods for solving shifted linear systems of equations based on the Arnoldi iteration using preconditioners that are also shifted systems. In section~\ref{sec:geneigen}, we discuss the convergence of the iterative solver and its connection to the convergence of the eigenvalues of the generalized eigenvalue problem $Kx=\lambda Mx$. In section~\ref{sec:inexact}, we discuss an error analysis when an iterative method is used to invert the preconditioner matrices. In section~\ref{sec:application}, we discuss the basic constitutive equations in OHT, which can be expressed as shifted linear system of equations and discuss the geostatistical method for solving inverse problems. Finally, in section~\ref{sec:numerical} we present some numerical results on systems of shifted systems and then discuss numerical results involving the inverse problem arising from OHT. We observe significant speed-ups using our Krylov subspace solver.

\section{Krylov subspace methods for shifted systems}\label{sec:krylov}
The goal is to solve systems of equations of the form 
\begin{equation}
 \label{eqn:multipleshifted}
  \left( K + \sigma_j M \right) x_j  = b \qquad j=1,\dots,n_f 
\end{equation}

Note that $\sigma_j$, for $j=1,\dots,n_f$ are (in general) complex shifts. We assume that none of these systems are singular. In particular, for our application, both $K$ and $M$ are stiffness and mass matrices respectively and are positive definite, but our algorithm only requires that they are invertible. By using a finite volume or lumped mass approach~\cite{hughes2012finite}, the mass matrices become diagonal but this assumption is not necessary. Later, in sections~\ref{sec:forward} and~\ref{sec:sensitivity}, we will show how such equations arise in our applications.

 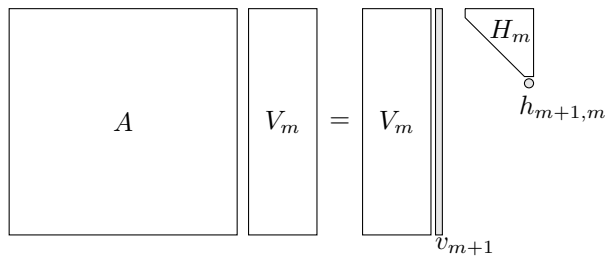
\begin{figure}[!ht]
  \centering
  \begin{tikzpicture}[scale = 0.3]
   \draw (0,0) rectangle ( 10, 10);
   \node (A) at (5,5) {$A$};
   \draw (10.5, 0) rectangle (13.5, 10);
   \node (Vm) at (12, 5) {$V_m$};
   \node (eq) at (14.5, 5) {$=$};
   \draw (15.5,0) rectangle (18.5, 10); 
   \node (Vm1) at (17,5) {$V_m$};
   \draw[fill=gray!20] (18.7,0) rectangle (19,10);
   \node (vm1) at (20,-0.5) {$v_{m+1}$};
   
   \draw (20,10) -- (23, 10) -- (23,7) -- (22.6, 7) -- (20,9.6) -- (20, 10) ;
   \node (Hm) at (22,9) {$H_m$}; 
  
   \draw[fill=gray!20] (22.8,6.7) circle [radius = 0.2];
   \node (hm1m) at (24.3,5.6) {$h_{m+1,m}$};
   
  \end{tikzpicture}
  \caption{Representation of the Arnoldi algorithm after $m$ steps.}
\label{fig:arnoldi}
 \end{figure}

As a brief introduction, we review the Krylov based iterative solvers for the system of equations $Ax = b$. In particular, we describe the variants generated by Arnoldi iteration, such as Full Orthogonalization Method (FOM) and Generalized Minimum RESidual method (GMRES). Krylov solvers typically generate a sequence of orthonormal vectors $v_1,\dots,v_m$ that are orthonormal. These vectors form a basis for the Krylov subspace: 
\[ \krylov{A,b}{m} \define \Span{b,Ab,\dots,A^{m-1}b}\] 
At the end of the $m$th iteration, a typical relation is obtained of the form (see figure~\ref{fig:arnoldi}),
\[ AV_m = V_{m+1}\bar{H}_m\]
where, $V_m = [v_1,\dots,v_m]$ and $v_1 = b/\beta$, $\beta = \normtwo{b}$ with $x_0= 0$ and $\bar{H}_m$ is an upper Hessenberg matrix. Then, an approximate solution to the system $Ax = b$ by searching for a solution of the form $x_m =  V_my_m$, where $y_m$ is chosen such that it minimizes the residual $r_m \define b - Ax_m$ which, leads to GMRES subproblem,
\begin{equation}\label{eqn:gmressubproblem} 
\min_{y_m \in \mathbb{C}^m}\normtwo{r_m}   \Rightarrow \min_{y_m \in \mathbb{C}^m} \normtwo{\beta e_1 - \bar{H}_my_m} 
\end{equation} 
  or an oblique projection $r_m \perp \Span{V_m}$ which leads to the FOM subproblem
\begin{equation}\label{eqn:fomsubproblem} 
r_m \perp \Span{V_m}  \qquad \Rightarrow \qquad H_my_m = \beta e_1 \end{equation}
As $m$ increases, the cost per iteration increases at least as ${\cal{O}}(m^2n)$ and the memory costs increase as ${\cal{O}}(mn)$\cite{saad2003iterative}. The standard remedies to reducing the number of iterations are 1) using an appropriate preconditioner, 2) truncating the orthogonalization in the Arnoldi algorithm and 3) restarting the Arnoldi algorithm periodically.

An interesting property of the Krylov subspaces is that they are shift-invariant. In other words, $\krylov{A,b}{m} =\krylov{A + \sigma I,b}{m}$. Therefore, the same Krylov basis generated for the system $Ax=b$ can be effectively used to solve shifted systems of the form $(A+\sigma I)x = b$. The strategy for solving the shifted systems is therefore, to first generate a basis that is applicable to all systems, and then use the shift-invariant property (for a detailed review, see \cite[section 14.1]{simoncini2007recent} and references therein) to solve a smaller subproblem of the form~\eqref{eqn:fomsubproblem} or~\eqref{eqn:gmressubproblem}. The same idea can be extended to systems of the form~\eqref{eqn:multipleshifted} using a preconditioner of the form $(K+\tau M)$~\cite{meerbergen2003solution,gu2007flexible}, which solves for multiple shifted systems roughly at the cost of a single system. However, in practice, the number of iterations taken can often be large, especially for large matrices arising from 
realistic 
applications.

In order to minimize the number of iterations, Meerbergen~\cite{meerbergen2003solution} proposes a left preconditioner of the form $K_\tau \define K + \tau M $ that is factorized and inverted using a direct solver. The application of $K_\tau^{-1}$ to a vector is, in general, not cheap but the spectrum of $(K+\tau M)^{-1}(K + \sigma M)$ is often more favorable, which results in fast convergence of the Krylov methods in just a few iterations~\cite{meerbergen2003solution}. This form of preconditioning has its roots in solving large-scale generalized eigenvalue problems and is known as {\it Cayley transformation}~\cite{golub1996matrix}. In~\cite{popolizio2008acceleration}, the authors provide some analysis for choosing the best value of $\tau$ that optimally preconditions all the systems. However, we observed that (also, see~\cite{gu2007flexible}) using a single preconditioner for all the systems may not yield optimal convergence for all systems. In~\cite{gu2007flexible}, the authors propose a flexible Arnoldi 
method for shifted systems that uses different values of $\tau$ resulting in different preconditioners at each iteration. This can potentially reduce the number of iterations for  all the shifted systems. Before we describe our flexible algorithm in section~\ref{sec:flexibleprecond}, we will derive the right preconditioned version of Krylov subspace method for shifted systems. This serves two purposes - it motivates our algorithm, while clarifying some of the notation.

\subsection{Right preconditioning for shifted systems}\label{sec:rightprecond}
As mentioned earlier, we will review the right preconditioned version of the Krylov subspace algorithm for shifted systems. Following the approach in~\cite{meerbergen2003solution,simoncini2007recent}, we solve the system of equations~\eqref{eqn:multipleshifted} using a shifted right preconditioner of the form $K_\tau \define K+\tau M $
\begin{equation}
 \label{eqn:rightpreconditioned}
  (K + \sigma_j M) K_\tau^{-1} \bar{x}(\sigma_j) = b \qquad x(\sigma_j) =  K_\tau^{-1}\bar{x}(\sigma_j)
\end{equation}
for $j=1,\dots,n_f$. We have the following identity that 
\begin{equation}\label{eqn:shiftinvariant}
 (K+\sigma M) (K + \tau M)^{-1} = I + (\sigma - \tau) M (K+\tau M)^{-1}
\end{equation}

Using the identity in equation~\eqref{eqn:shiftinvariant}, we have the following shift-invariance property $\krylov{MK_\tau^{-1},b}{m} = \krylov{(K+\sigma M)(K+\tau M)^{-1},b}{m}$. Note that $\krylov{MK_\tau^{-1},b}{m} $ is independent of $\sigma$. This shift-invariance property suggests an efficient algorithm for solving the system of equations~\eqref{eqn:rightpreconditioned}. There is a distinct advantage in using iterative solvers for shifted systems; the expensive step of constructing the basis for the Krylov subspace is performed only once and using the shift-invariance property of the Krylov subspace, the sub-problem for each shift in algorithm~\ref{alg:shiftedsolve} can be computed at a relatively low cost. 

\begin{algorithm}[!ht]
 \begin{algorithmic}[1]
  \REQUIRE $M$ and a right hand side $b$.
  \STATE Compute $v_1 = b /\beta$ and $\beta \define \normtwo{b}$
  \STATE Choose $\tau$ and factorize $K_\tau \define K + \tau M$
  \STATE Define the $(m+1)\times m$ matrix  $\bar{H}_m = \{ h_{i,k}\}_{1\leq i \leq m+1, 1 \leq k \leq  m}$. Set $\bar{H}_m = 0$.
  \FORALL {$k=1,\dots,m$}
  \STATE Compute ${z}_k = K_\tau^{-1} v_k$ 
  \STATE $w_k := Mz_k$
  \FORALL {$i=1,\dots,k$} 
  \STATE $h_{ik} := w_k^*v_i$
  \STATE Compute $w_k := w_k - h_{i,k}v_k$
  \ENDFOR 
  \STATE $h_{k+1,k}:= \lVert w_k \rVert_2 $. If $h_{k+1,k} = 0$ stop 
  \STATE $v_{k+1} = w_k/ h_{k+1,k}$
   \ENDFOR
 \end{algorithmic}
\caption{Arnoldi using Modified Gram-Schmidt~\cite{saad2003iterative}:}
\label{alg:arnoldi}
\end{algorithm}

The algorithm proceeds as follows: first, we run $m$ steps of the Arnoldi algorithm on the matrix $M K_\tau^{-1}$ with the starting vector $b$ to get a basis for the Krylov subspace $\krylov{MK_\tau^{-1},b}{m}$. This is summarized in algorithm~\ref{alg:arnoldi}. At the end of $m$ steps of the Arnoldi process, we construct two sets of vectors $V_{m+1} = [v_1,\dots,v_{m+1}]$ and $Z_m = [z_1,\dots,z_m ]$, and an upper Hessenberg matrix $\bar{H}_m$ that satisfy the following relations,

\begin{align}
\label{eqn:arnoldi}
 M  Z_m  = & \quad V_{m+1}\bar{H}_m \\
   (K + \tau M) Z_m =  & \quad V_m  \label{eqn:zm}
\end{align}
where, $ V_m^* V_m = I $. Multiplying the first equation by $(\sigma_j-\tau)$ and adding it to the second equation gives us 

\begin{equation}
 \label{eqn:shiftedarnoldi}
 (K + \sigma_j M)Z_m = V_{m+1} \underbrace{\left(\begin{bmatrix} I \\ {0}\end{bmatrix} + (\sigma_j-\tau) \bar{H}_m  \right)}_{\define \bar{H}_m(\sigma_j; \tau)} = V_{m+1}\bar{H}_m(\sigma_j;\tau)
 \end{equation}
In algorithm~\ref{alg:arnoldi}, $V_m$ forms a basis for the  Krylov subspace ${\cal{K}}_m(MK_\tau^{-1},b)$.
However, we seek solutions of the form $x_m =  Z_my_m$ (with zero as the initial guess). Now $x_m \in \Span{Z_m}$ where $Z_m$ is the  space spanned by the vectors $z_k = K_\tau^{-1}v_k$ for $k=1,\dots,m$. By minimizing the residual norm over all possible vectors in $\Span{Z_m}$, we obtain the generalized minimum residual (GMRES) method for shifted systems, whereas by imposing the Petrov-Galerkin condition $r_m \perp\Span{V_m}$, we obtain the full orthogonalized method (FOM) for shifted systems. This is summarized in algorithm~\ref{alg:shiftedsolve}.  It should be noted that the way we have described this algorithm, we need to store the vectors $Z_m$. In practice, this is not necessary. We chose to present it this way in order to have consistent notation with the flexible algorithm we will describe in subsection~\ref{sec:flexibleprecond}.

\begin{algorithm}[!ht]
 \begin{algorithmic}[1]
 \REQUIRE matrices $K$ and $M$, a right hand side $b$, $\sigma \in \{ \sigma_1,\dots,\sigma_{n_f}\}$
 \STATE Choose $\tau$, build $K_\tau \define K + \tau M$ and construct preconditioner. Set $m = 1$.
 \WHILE { all systems have not converged}
   \STATE Generate $V_{m+1},\bar{H}_m$ and $Z_m$ using algorithm~\ref{alg:arnoldi}.
  \FORALL {$j = 1,\dots,n_f$}
   \IF {system $j$ not converged} 
   \STATE Construct $\bar{H}_m(\sigma_j;\tau) \define I + (\sigma_j -\tau)\bar{H}_m$ (see equation~\eqref{eqn:shiftedarnoldi}).
   \STATE FOM: 
    \[  H_m(\sigma_j; \tau)y_m^{fom}(\sigma_j) =  \beta e_1\]
   \STATE GMRES: 
    \[ y_m^{gmres}(\sigma_j) \define \min_{ y_m \in \mathbb{C}^m} \lVert \beta e_1 - \bar{H}_m(\sigma_j;\tau) y_m \rVert_2\] 
    \STATE Construct the approximate solution  $x_m(\sigma_j) = Z_my_m(\sigma_j)$
	\ENDIF	
	\ENDFOR
	\STATE $m\leftarrow m + 1$. 
  \ENDWHILE
 
 \end{algorithmic}
\caption{FOM/GMRES for Shifted Systems}
 \label{alg:shiftedsolve}
\end{algorithm}

In~\cite{meerbergen2003solution}, the spectrum of $(K+\sigma M)(K+\tau M)^{-1}$ was analyzed and it was shown that the preconditioner $ K_\tau$ is well suited only for values of frequencies $\sigma$ near $\tau$. However, the values of $\sigma$ can be widely spread and a single preconditioner $K_\tau$ might not be a good choice for preconditioning all the systems. In section~\ref{sec:solverexperiments}, we demonstrate an example in which a single preconditioner does not satisfactorily precondition all the systems. In~\cite{gu2007flexible}, the authors propose a flexible approach using a (possibly) different preconditioner at each iteration. We shall adopt this 
approach.

\subsection{Flexible preconditioning}\label{sec:flexibleprecond}

We now describe our flexible Krylov approach for solving shifted systems based on~\cite{gu2007flexible} which we have extended to the case $M\neq I$. Following~\cite{saad1993flexible} and~\cite{gu2007flexible}, we use a variant of GMRES which allows a change in the preconditioner at each iteration. In algorithm~\ref{alg:arnoldi}, we considered a fixed preconditioner of the form $K_\tau \define K + \tau M$ for a fixed $\tau$. Suppose we used a different preconditioner at each iteration of the form $K+\tau_k M$ for $k=1,\dots,m$, then instead of~\eqref{eqn:zm} we have, 
\begin{equation}
   (K+\tau_k M) z_k =  v_k  \qquad k=1,\dots,m \label{eqn:zmmod}
\end{equation}
The algorithm is summarized in algorithm~\ref{alg:shiftedsolvemod}. In this algorithm, in addition to saving $V_m$, we also save the matrix $Z_m$. If at every step in the flexible Arnoldi algorithm we use the same value of $\tau$, we are in the same position as in algorithm~\ref{alg:arnoldi}. We have $Z_m = [z_1,\dots,z_m]$, $\bar{H}_m = \{h_{ik}\}_{1\leq i \leq m+1,1\leq k \leq m}$ and $V_{m} = [v_1,\dots,v_{m}]$ which satisfies $V_m^*V_m = I_m$. In addition, we also have the following relations 

\begin{align} 
 MZ_m = & \quad V_{m+1}\bar{H}_m \label{eqn:arnoldimod} \\
 KZ_m + MZ_m T_m = & \quad V_m   \label{eqn:zmvmrelation}
\end{align}
where, $T_m = \text{diag}\{ \tau_1,\dots,\tau_m\}$. Multiplying~\eqref{eqn:arnoldimod} by $\sigma_jI_m - T_m$ and adding~\eqref{eqn:zmvmrelation}, we obtain for $j = 1,\dots,n_f$  
\begin{equation}
(K + \sigma_jM )Z_m = V_{m+1} \underbrace{\left( \begin{bmatrix} I \\ 0 \end{bmatrix}  + \bar{H}_m (\sigma_jI_m - T_m) \right)}_{\define \bar{H}(\sigma_j; T_m)} = V_{m+1}\bar{H}_m(\sigma_j;T_m) \label{eqn:shiftedarnoldimod} 
\end{equation}

\begin{algorithm}[!ht]
 \begin{algorithmic}[1]
  \REQUIRE $M$ and $b$ the right hand side, $\tau_k, k=1,\dots,m$, $v_1 = b /\beta$ and $\beta \define \normtwo{b}$
  \STATE Define the $(m+1)\times m$ matrix  $\bar{H}_m = \{h_{i,k} \}_{1\leq i \leq m+1, 1 \leq k \leq  m}$. 
  \FORALL {$k=1,\dots,m$}
  \STATE Solve $ (K+\tau_kM) {z}_k =  v_k$
  \STATE $w_k := Mz_k$
   \FORALL {$i=1,\dots,k$} 
    \STATE $h_{i,k} := w_k^*v_i$
    \STATE Compute $w_k := w_k - h_{i,k}v_k$
  \ENDFOR 
  \STATE $h_{k+1,k}:= \lVert w_k \rVert_2 $. If $h_{k+1,k} = 0$ stop 
  \STATE $v_{k+1} = w_k/ h_{k+1,k}$
   \ENDFOR
    \end{algorithmic}
\caption{Flexible Arnoldi using Modified Gram-Schmidt }
\label{alg:arnoldimod}
\end{algorithm}

We are now in a position to derive a FOM/GMRES algorithm for shifted systems with flexible preconditioning. We search for solutions which are approximations of the form $x_m(\sigma_j) = Z_m y_m(\sigma_j)$, which spans the columns of $Z_m$. Strictly speaking, $\text{span}\{Z_m\}$ is no longer a Krylov subspace. By minimizing the residual norm over all possible vectors in $\Span{Z_m}$, we obtain the flexible generalized minimum residual (FGMRES) method for shifted systems, whereas by imposing the Petrov-Galerkin condition $r_m \perp\Span{V_m}$, we obtain the flexible full orthogonalized method (FFOM) for shifted systems. This is summarized in algorithm~\ref{alg:shiftedsolvemod}. The residuals can be computed as 
\begin{align}\label{eqn:fomgmresresidual}
r_m(\sigma_j) = & \quad b - (K+\sigma_jM)x_m(\sigma_j) \\ \nonumber
              = & \quad V_{m+1}\left(\beta e_1 - \bar{H}_m(\sigma_j;T_m)y_m(\sigma_j)\right) \\ \nonumber
\end{align}

\begin{algorithm}[!ht]
 \begin{algorithmic}[1]
 \REQUIRE matrices $K$ and $M$, vector $b$, $\sigma \in \{ \sigma_1,\dots,\sigma_{n_f}\}$, set $m=1$.

 \WHILE {all systems have not converged}
\STATE  Choose $T_m = \text{diag}\{\tau_1,\dots,\tau_m\}$.
  \STATE Generate $V_{m+1},\bar{H}_m$ and $Z_m$ using algorithm~\ref{alg:arnoldimod}.
  \FORALL {$j= 1,\dots,n_f$}
   \IF {system $j$ has not converged} 
   \STATE Construct $\bar{H}_m(\sigma_j;T_m)\define I + \bar{H}_m(\sigma_jI-T_m)$ (see equation~\eqref{eqn:shiftedarnoldimod}).
   \STATE FOM: 
    \[  H_m(\sigma_j; T_m)y_m^{fom}(\sigma_j) =  \beta e_1\]
   \STATE GMRES: 
    \[ y_m^{gmres}(\sigma_j) \define \min_{ y_m \in \mathbb{C}^m} \lVert \beta e_1 - \bar{H}_m(\sigma_j;T_m) y_m \rVert_2\] 
    \STATE Construct the approximate solution as $x_m(\sigma_j) = Z_m y_m(\sigma_j)$
  \ENDIF
 \ENDFOR
 \STATE $m\leftarrow m+1$
\ENDWHILE 
 \end{algorithmic}
\caption{Flexible FOM/GMRES for Shifted Systems}
 \label{alg:shiftedsolvemod}
\end{algorithm}

\subsubsection{Selecting values of $\tau_k$} 
In algorithm~\ref{alg:arnoldimod}, at each iteration we solve a system of the form $(K+\tau_k M)z_k = v_k$ for $k=1,\dots,m$. This cost can be high if the dimension of the Arnoldi subspace $m$ is large and a different preconditioner $K+\tau_kM$ is used at every iteration. In practice, it is not necessary to form and factorize $m$ systems corresponding to different $\tau_k$. In applications, we only need choose a few different $\tau_k$ that cover the entire range of the parameters $\sigma_j$. This was also described in~\cite{gu2007flexible}. The system of equations~\eqref{eqn:zmmod} is solved using a direct solver and since only a few values of $\tau_k$ are chosen, the systems can be formed and factorized. Thus, the computational cost will not be affected greatly even if the number of frequencies $n_f$ is large.

Let $\bar{\tau} = \{\bar{\tau}_1,\dots,\bar{\tau}_{n_p} \}$ be the set of values that $\tau_k$ can take. In other words, we take $n_p$ distinct preconditioners. Then, the first $m_1$ values of $\tau_k$ are assigned $\bar{\tau}_1$, the next $m_2$ values of $\tau_k$ are assigned $\bar{\tau}_2$ and so on. We also have $m = \sum_{k=1}^{n_p}m_k$. 

\subsubsection{Restarting}\label{sec:restarting}

As the dimension of the subspace $m$ increases, the computational and memory costs increase significantly. A well known solution to this problem is restarting. The old basis is discarded and the Arnoldi algorithm is restarted on a new residual. However, for shifted systems, in order to preserve the shift-invariant property, one needs to ensure collinearity of the residuals of the shifted systems. For FOM the residuals are naturally collinear and the Arnoldi algorithm can be restarted by scaling each residual by some scalar that depends on the shift~\cite{simoncini2003restarted,gu2007flexible}. For GMRES, the approach used by~\cite{frommer1998restarted} was extended to shifted systems with multiple preconditioners by~\cite{gu2007flexible}. We did not explore this issue further, and the reader is referred to~\cite{gu2007flexible} for further details.

\section{Generalized eigenvalue problem and error estimates}\label{sec:geneigen}
We start by computing the approximate eigenvalues and eigenvectors for the matrix $(K+\sigma M)M^{-1}$. Using estimates for approximate eigenvalues and eigenvectors we derive expressions for the convergence of the flexible algorithms. The approximate eigenvalues are called Ritz values. For convenience, we drop the subscript on the shifted frequency, i.e., use $\sigma$ instead of $\sigma_j$ where, $j=1,\dots,n_f$.  
\begin{propos}\label{prop:ritz}
Let $Z_m$, $\bar{H}_m$ and $V_{m+1}$ be computed according to algorithm~\ref{alg:arnoldimod}. Calculate the eigenpairs of the generalized eigenvalue problem
\begin{equation}\label{eqn:genritz}
H_m(\sigma;T_m) f  = \theta  H_m f  
\end{equation}
Then, the Ritz pair $\left( \theta, u \define V_{m+1}\bar{H}_mf \right)$ satisfy the Petrov-Galerkin condition~\cite[section 4.3.3]{saad1992numerical}
\begin{equation}\label{eqn:petrovgalerkin}
(K + \sigma M)M^{-1}u - \theta u \perp \Span{V_m} \qquad u \in \Span{V_{m+1}\bar{H}_m}
\end{equation}
\end{propos}
\begin{proof}
We first begin by manipulating equations~\eqref{eqn:arnoldimod} and~\eqref{eqn:zmvmrelation}. Eliminating $Z_m$ from those equations and adding $\sigma V_{m+1}\bar{H}_m$ to both sides, we have 
\begin{equation}\label{eqn:rationalarnoldi}
V_{m+1}\bar{H}_m(\sigma;T_m) =  (K+\sigma M)M^{-1}V_{m+1}\bar{H}_m
\end{equation}
Now, consider the residual of the eigenvalue calculation for the $k$th eigenpair, where $k=1,\dots,m$ is

\begin{align}\label{eqn:residualritzeigen}
 r^\text{eig}_k(\sigma) = &\quad  (K+\sigma M)M^{-1}V_{m+1}\bar{H}_m f_k - \theta_k V_{m+1}\bar{H}_mf_k  \\ \nonumber
           = & \quad  V_{m+1}\bar{H}_m(\sigma;T_m)f_k - \theta_k V_{m+1}\bar{H}_mf_k  \\ \nonumber
	   = &\quad  V_m\left(H_m(\sigma;T_m)f_k - \theta_kH_mf_k \right) - h_{m+1,m}v_{m+1}(\tau_m +\theta_k - \sigma)e_m^*f_k 	 \\ \nonumber
	   = & \quad - h_{m+1,m}v_{m+1}(\tau_m +\theta_k - \sigma)e_m^*f_k 	 \\ \nonumber
\end{align}
From which we can claim that $u \in \Span{V_{m+1}\bar{H}_m}$ and $(K+\sigma M)M^{-1}u - \theta u \perp \Span{V_m}$. In other words, they satisfy the Petrov-Galerkin~\eqref{eqn:petrovgalerkin} and are an approximate eigenpair of $(K+\sigma M)M^{-1}$.
\end{proof}

Furthermore, we define $\rho_k \define \normtwo{(K+\sigma M)M^{-1}u_k-\theta_k u_k}$, which is the residual norm of the $k$th eigenvalue calculations. When the residual of the eigenvalue calculations $\rho_k$ is small, say machine precision,  the Ritz values are a good approximation to the eigenvalues. It is readily verified that the eigenvalues $\lambda$ of $KM^{-1}$ (and the generalized eigenvalue problem $Kx=\lambda Mx$) are related to the eigenvalues $\lambda(\sigma)$ of $(K+\sigma M)M^{-1}$ by the relation $\lambda(\sigma) = \lambda + \sigma$. The importance of the convergence of Ritz values to the convergence of the Krylov subspace solver using FOM can be established by the following result. 

\begin{propos}\label{prop:fom}
Assume the requirements of proposition~\ref{prop:ritz}. Further, assume $F$ (matrix of generalized eigenvectors, see~\eqref{eqn:genritz}) is invertible so that the generalized eigendecomposition $H_m(\sigma;T_m) = H_mF\Theta F^{-1}$ exists. The residual using FFOM satisfies the following inequality
\begin{equation}\label{eqn:fomresidualineq} \normtwo{r_m(\sigma)} \leq \sum_{k=1}^m \rho_k\left|\frac{\sigma -\tau_m}{\theta_k +\tau_m-\sigma}\right| |\theta_k^{-1}| |s_k| \end{equation}
where, $s_k \define e_k^*F^{-1}(H_m^{-1}\beta e_1)$ and $\rho_k$ is the residual norm of the eigenvalue calculation, defined above. 
\end{propos}
\begin{proof}
We start by writing the residual in equation~\eqref{eqn:fomgmresresidual}
\[ r_m(\sigma) =  V_m( \beta e_1 - H_m(\sigma;T_m)y_m) - v_{m+1}(\sigma-\tau_m)h_{m+1,m}e_m^*y_m\]
Since, for flexible FOM for shifted systems $y_m = H_m(\sigma;T_m)^{-1}\beta e_1$, the first term in the above expression is zero and we have,
\begin{equation}\label{eqn:fomresidualsimplified}
r_m(\sigma) =  \quad  -v_{m+1}(\sigma-\tau_m)h_{m+1,m}e_m^*H_m(\sigma;T_m)^{-1}\beta e_1  \\ \nonumber
\end{equation}
 Now, using the generalized eigendecomposition in~\eqref{eqn:genritz} $H_m(\sigma;T_m) = H_m F\Theta F^{-1}$. Therefore, we have $H_m(\sigma;T_m)^{-1} = F\Theta^{-1} F^{-1}H_m^{-1}$. We can write this is as a sum of rank-$1$ vectors 
\[ H_m(\sigma;T_m)^{-1} = \sum_{k=1}^m \theta_k^{-1}f_k e_k^*F^{-1}H_m^{-1} \]
where, $e_k$ is the $k$-th canonical basis vector and $f_k$ is the $k$-th column of $F$ for $k=1,\dots,m$. Using the residual of the eigenvalue calculation $r^\text{eig}_k(\sigma)$ in~\eqref{eqn:residualritzeigen} and the expression derived above, 
\begin{align} 
r_m(\sigma) \quad = & \quad -\sum_{k=1}^mv_{m+1}h_{m+1,m} (\sigma -\tau_m) e_m^*f_k\theta_k^{-1}s_k \\ \nonumber
		= & \quad \sum_{k=1}^m r^\text{eig}_k(\sigma) \frac{\sigma - \tau_m}{\tau_m +\theta_k - \sigma} \theta_k^{-1} s_k \\ \nonumber
\end{align}
where, $s_k\define e_k^*F^{-1}H_m^{-1}\beta e_1$. The proof follows from the properties of vector norms.
\end{proof}

The inequality~\eqref{eqn:fomresidualineq} provides insight into the importance of the accuracy of approximate eigenpairs for the convergence of flexible FOM for shifted systems. We follow the arguments in~\cite{meerbergen2003solution}. In particular, the residual is very small if $\sigma \approx \tau_m$, $|\theta_k^{-1}|$, $|s_k|$ or $\rho_k$ are small. We shall ignore the case that $\sigma \approx \tau_m$ for further analysis, i.e. that the shifted system is almost exactly the preconditioned system. The eigenvalue residual norm $\rho_k$ being small implies that the Ritz values are a good approximation to the eigenvalues of $(K+\sigma M)M^{-1}$. This implies that all the eigenvalues in this interval have been computed fairly accurately. We now discuss when $|\theta_k^{-1}|$ is large. When all the values of $\tau_k$ are equal to $\tau$, the approximate eigenvalues $\theta_k$ of $KM^{-1}$ are related to approximate eigenvalues $\lambda_k$ of the preconditioned system $(K+\sigma M)(K+\tau M)^{-1}$ by the 
Cayley transformation $\frac{\lambda_k + \sigma}{\lambda_k + \tau}$. Therefore, $|\theta_k^{-1}|$ is large only if $|\lambda_k + \sigma| \ll |\lambda_k + \tau|$. The term $s_k$ can be rewritten as $s_k =  e_k^*F^{-1}H_m^{-1}V_m^*V_m\beta e_1 = e_k^*F^{-1}H_m^{-1}V_m^* b$. It is readily verified that $e_k^*F^{-1}H_m^{-1}V_m^*$ is orthonormal to all other approximate eigenvectors $V_{m+1}\bar{H}_mZ_m$ and thus, $s_k$ can be interpreted as the component of the right hand side $b$ in the direction of the approximate eigenvector. In other words, $s_k$ is small when the solution $x_m(\sigma)$ has a small component in the direction of $b$.

The analysis for the convergence of flexible FOM for shifted systems can be extended to flexible GMRES as well. The following result bounds the difference in the residuals obtained from $m$ steps using flexible FOM and flexible GMRES. 
\begin{propos}\label{prop:gmres}
 Let $Z_m$, $\bar{H}_m$ and $V_{m+1}$ be computed according to algorithm~\ref{alg:arnoldimod}. Further, from algorithm~\ref{alg:shiftedsolvemod} we define the flexible FOM quantities $y_m^\text{fom}(\sigma) = H_m(\sigma;T_m)^{-1}\beta e_1$, residual $r_m^\text{fom}(\sigma) = V_{m+1}(\beta e_1-\bar{H}_m(\sigma;T_m)y_m^\text{fom}(\sigma))$ and flexible GMRES quantities  $y_m^\text{gmres}(\sigma) = \arg\min_{y\in \mathbb{C}^m}\normtwo{\beta e_1-\bar{H}_m(\sigma;T_m)y}$, residual $r_m^\text{gmres}(\sigma) = V_{m+1}(\beta e_1-\bar{H}_m(\sigma;T_m)y_m^\text{gmres}(\sigma))$. Further, assume that $H_m(\sigma;T_m)$ is invertible. We have the following inequality 
\begin{equation}\label{eqn:fomgmresresdiff}
\normtwo{r_m^\text{fom}(\sigma) - r_m^\text{gmres}(\sigma)} \leq  \frac{\alpha(1+\alpha)}{1+\alpha^2}\normtwo{r_m^\text{fom}} 
\end{equation} 
where, $\eta \define h_{m+1,m}(\sigma - \tau_m)$ and $\alpha \define  \normtwo{\eta H_m^{-*}(\sigma;T_m)e_m}$.
\end{propos}

\begin{proof} We begin by the following observation from equation~\eqref{eqn:fomresidualsimplified} $r_m^\text{fom}(\sigma) = -\eta e_m^*y_m^\text{fom} v_{m+1}$ and 
\[r_m^\text{fom}(\sigma) - r_m^\text{gmres}(\sigma) = -V_{m+1}\bar{H}_m(\sigma;T_m) \left( y_m^\text{fom}(\sigma) - y_m^\text{gmres} (\sigma)\right) \]
Next, we look at the solution to the GMRES least squares problem which can be written as the normal equations 
\[ \bar{H}_m^*(\sigma;T_m)\bar{H}_m(\sigma;T_m)y_m^\text{gmres} = \bar{H}_m^*(\sigma;T_m)\beta e_1 = H_m^*(\sigma;T_m)\beta e_1\]
This can be rewritten as 
\begin{align*}
 \left(H_m^*(\sigma;T_m)H_m(\sigma;T_m)+\eta^2e_me_m^*\right)y_m^\text{gmres}(\sigma) =& \quad  H_m^*(\sigma;T_m)\beta e_1 \\ \nonumber
\left(H_m(\sigma;T_m) + \eta^2H_m^{-*}e_me_m^*\right)y_m^\text{gmres}(\sigma) = & \quad \beta e_1 
\end{align*}
In other words, the solution to the GMRES subproblem is a rank-one perturbation of the FOM subproblem. Using the Sherman-Morrison identity 
\[ y^\text{gmres}_m(\sigma) = \underbrace{H_m(\sigma;T_m)^{-1}\beta e_1}_{ =  y^\text{fom}_m(\sigma) } - H_m(\sigma;T_m)^{-1}\eta H_m^{-*}(\sigma;T_m)e_m \frac{(\eta e_m^*H_m(\sigma;T_m)^{-1}\beta e_1)}{ 1+  \normtwo{\eta H_m^{-*}(\sigma;T_m)e_m}^2}\] 
Then, the residual difference between FOM and GMRES can be bounded as 
\[\normtwo{r_m^\text{fom}(\sigma) - r_m^\text{gmres}(\sigma)}  \leq \normtwo{\bar{H}_m(\sigma;T_m)H_m(\sigma;T_m)^{-1}} \frac{\alpha|\eta e_m^*y_m^\text{fom}(\sigma)|}{1+  \alpha^2}\]
The inequality~\eqref{eqn:fomgmresresdiff} follows from the following observations $\normtwo{\bar{H}_m(\sigma;T_m)H_m(\sigma;T_m)^{-1}} \leq 1 +  \normtwo{\eta H_m^{-*}(\sigma;T_m)e_m}$ and from~\eqref{eqn:fomresidualsimplified} we have $\normtwo{r_m^\text{fom}}  =  |\eta e_m^*y_m^\text{fom}(\sigma)|$.
\end{proof}
 
If $\normtwo{\eta H_m^{-*}(\sigma;T_m)e_m}$ is large, then the difference between the two residuals can be large. This happens either when $\eta$ is large or $H_m(\sigma;T_m)$ is close to singular. In this case, flexible GMRES can stagnate and further progress may not occur. We now discuss situations in which breakdown occurs, i.e. $h_{m+1,m} =0$. If $h_{m+1,m} \neq 0$ and $H_m$ is full rank, then it can be shown from equation~\eqref{eqn:arnoldimod} that $\Span{MZ_m} \subseteq \Span{V_{m+1}} $ and from equation~\eqref{eqn:shiftedarnoldimod}, it follows that $\Span{(K+\sigma_j M)Z_m} \subseteq \Span{V_{m+1}}$. Further, $h_{m+1,m} = 0$ if and only if $x_m(\sigma_j)$ is the exact solution and $H_m(\sigma_j;T_m)$ is non-singular. The argument closely follows~\cite{saad1993flexible} and will not be repeated here.





\section{Inexact preconditioning}\label{sec:inexact}

We observe that to compute vectors $z_k$ for $k =1,\dots,m$ in equation~\eqref{eqn:zmmod}, we have to invert matrices of the form $K+\tau_kM$. When the problem sizes are large, iterative methods may be necessary to invert such matrices, resulting in a variable preconditioning procedure in which a different preconditioning operator is applied at each iteration. More precisely, for $k=1,\dots,m$,
\begin{equation} \label{eqn:zmmoditer}
\tilde{z}_k \approx (K + \tau_k M)^{-1}v_k \qquad p_k \define v_k - (K +\tau_k M) \tilde{z}_k
\end{equation}
where, $p_k$ is the residual that results after the iterative solver has been terminated. To simplify the discussion, we assume that the termination criteria for the iterative solver is such that $\normtwo{p_k} \leq \varepsilon \underbrace{\normtwo{v_k}}_{= 1} = \varepsilon$, for some $\varepsilon$. We  closely follow the approach in~\cite{simoncini2003theory}. The new flexible Arnoldi relationship is now, 

\begin{equation}
\label{eqn:flexshiftedarnoldimod}
(K + \sigma_j M)\tilde{Z}_m  + P_m = V_{m+1}\bar{H}_m(\sigma_j;T_m)  \qquad j=1,\dots,n_f
\end{equation} 
where, $\tilde{Z}_m = [\tilde{z}_1,\dots,\tilde{z}_m]$ and $P_m = [p_1,\dots,p_m]$ and $\bar{H}_m(\sigma_j;T_m)$ is defined in equation~\eqref{eqn:shiftedarnoldimod}. By using inexact applications of the preconditioner, the vectors $v_k$ for $k=1,\dots,m+1$ are no longer the same vectors generated from algorithm~\ref{alg:arnoldimod}. In particular, $\Span{V_m}$ is no longer a Krylov subspace generated by $A$. However, by construction, $V_m$ is still an orthogonal matrix.

Having constructed the matrix $\tilde{Z}_m$, we seek approximate solutions spanned by the columns of $\tilde{Z}_m$, i.e., solutions of the form $x_m(\sigma_j) = \tilde{Z}_m y_m(\sigma_j)$. The true residual corresponding to the approximation solution  $x_m(\sigma_j) = \tilde{Z}_m y_m(\sigma_j)$ can be computed as follows,
\begin{align*}
r_m(\sigma_j)\quad   = &  \quad b - (K+\sigma_jM)\tilde{Z}_my_m(\sigma_j) \\ 
               = &  \quad b - V_{m+1}\bar{H}_m(\sigma_j;T_m)y_m(\sigma_j) + P_m y_m(\sigma_j) \\ 
	       = &  \quad V_{m+1}\left(\beta e_1 - \bar{H}_m(\sigma_j;T_m)y_m(\sigma_j)\right) + P_my_m(\sigma_j) 
\end{align*}

The columns of the matrix $P_m$ are not computed in practice because they require an additional matrix-vector product with $K+\tau_k M$. As a result, computing the true residual is expensive. However, in order to monitor the convergence of the iterative solver, we need bounds on the true residual. Using such bounds, we can derive stopping criteria for the flexible Krylov solvers for shifted systems with inexact preconditioning. To do this, we first derive bounds on the norm of inexact residual $\tilde{r}_m(\sigma_j)$ and a bound on the difference between the true and the inexact residual $\normtwo{r_m(\sigma_j)-\tilde{r}_m(\sigma_j)}$. A simple application of the triangle inequality for vector norms, leads us to the desired bounds on the true residual

The inexact residual $ \tilde{r}_m(\sigma_j)  $ defined as 
\[ \tilde{r}_m(\sigma_j)  \define V_{m+1}\left(\beta e_1 - \bar{H}_m(\sigma_j;T_m)y_m(\sigma_j)\right)  \] 
The expression for$ \tilde{r}_m(\sigma_j)$ is similar to the exact residual  $\tilde{r}_m(\sigma_j)$ ignoring the error due to early termination of the inner iterative solver, i.e., $P_my_m(\sigma_j)$. It is easy to verify that $\normtwo{\tilde{r}_m(\sigma_j) } = \normtwo{\beta e_1 - \bar{H}_m(\sigma_j;T_m)y_m(\sigma_j)} $. We now derive an expression for the norm of the difference between the true and the inexact residuals,

\begin{align}
\normtwo{r_m(\sigma_j)-\tilde{r}_m(\sigma_j)} & =  \normtwo{P_my_m(\sigma_j)} =  \normtwo{\sum_{k=1}^m e_k^Ty_m(\sigma_j)p_k}  \nonumber\\
					& \leq     \sum_{k=1}^m|e_k^Ty_m(\sigma_j)|\normtwo{p_k}  \nonumber \\
					& \leq    \varepsilon\sum_{k=1}^m|e_k^Ty_m(\sigma_j)| = \varepsilon \lVert y_m(\sigma_j)\rVert_1 \label{eqn:resdiff}
\end{align}
Finally, the norm of the true residual $r_m(\sigma_j)$ can be bounded using the following relation
\begin{align}
\normtwo{r_m(\sigma_j)} \leq & \quad \normtwo{r_m(\sigma_j)-\tilde{r}_m(\sigma_j)} + \normtwo{\tilde{r}_m(\sigma_j)}  \nonumber \\
			\leq &  \quad \varepsilon \norm{y_m(\sigma_j)}{1}  + \normtwo{\beta e_1 - \bar{H}_m(\sigma_j;T_m)y_m(\sigma_j)}  \label{eqn:trueresidual}
\end{align}

This bound on the true residual, gives us a convenient expression to monitor the convergence of the iterative solver for each system, corresponding to a given shift $\sigma_j$.

We can also derive specialized results for the flexible FOM/GMRES for shifted systems with inexact preconditioning. The approach used is and argument similar to~\cite[Proposition 4.1]{simoncini2003theory}. Let $r^\text{fom}_m(\sigma_j) \define b - (K+\sigma_j M)\tilde{Z}_my_m^\text{fom}(\sigma_j)$ and $r^\text{gmres}(\sigma_j) \define b - (K+\sigma_j M)\tilde{Z}_my_m^\text{gmres}(\sigma_j)$ be the true residual, respectively resulting from the flexible FOM/GMRES for shifted systems. We have the following error bounds

\begin{align*}
\normtwo{V_m^*r_m^\text{fom}(\sigma_j)} \leq & \quad \varepsilon \norm{y_m^\text{fom}(\sigma_j)}{1} \\ 
\normtwo{\left(V_{m+1}\bar{H}_m(\sigma_j;T_m)\right)^*r_m^\text{gmres}} \leq & \quad \varepsilon \normtwo{\bar{H}_m(\sigma_j;T_m)} \norm{y_m^\text{fom}(\sigma_j)}{1}  
\end{align*}

One of main results of the paper~\cite{simoncini2003theory} is that they provide theory for why the residual norm due inexact preconditioning can be allowed to grow at the later outer iterations. In particular, they provide computable bounds for the monitoring the outer Krylov solver residual when the termination criteria for the inner preconditioning is allowed to change at each  iteration, from which efficient termination criteria can be derived. We have not pursued this issue and the reader is referred to~\cite{simoncini2003theory} for further details.

\section{Application to Oscillatory Hydraulic Tomography}\label{sec:application}

In this section, we briefly review the application of Oscillatory Hydraulic Tomography and the Geostatistical approach for solving the resulting inverse problem. 
\subsection{The Forward Problem}\label{sec:forward}

The equations governing ground water flow through an aquifer for a given domain $\Omega$ with boundary $\partial \Omega = \partial \Omega_D \cup \partial \Omega_N, \partial \Omega_D \cap \partial \Omega_N = \emptyset$ 
are given by,

\begin{align} \label{eqn:timedomain}
S_s(\bx) \frac{\partial \phi(\bx,t)}{\partial t} - \nabla \cdot \left(K(\bx) \nabla \phi(\bx,t)\right) & =   q(\bx,t), & \bx &\in \Omega \\ \nonumber
\phi(\bx,t) & = 0, & \bx & \in \partial \Omega_D  \\ \nonumber 
\nabla \phi(\bx,t) \cdot \textbf{n} & =  0, & \bx &\in \partial \Omega_N \\ \nonumber 
\end{align}
where $S_s(\bx)$ [L$^{-1}$] represents the specific storage and $K (\bx)$ [L/T] represents the hydraulic conductivity. In the case of 
one source oscillating at a fixed frequency $\omega$ [radians/T] , $q(\bx,t)$ is given by 
\begin{equation} \label{eq:oscillation}
 q(\bx,t) = Q_0\delta(\bx-\bx_s) \cos(\omega t) 
 \end{equation}
To model periodic simulations, we will assume the source to be a point source oscillating at a known frequency $\omega$ and peak amplitude $Q_0$ at the source location $\bx_s$. In the case of multiple sources oscillating at distinct frequencies, each source is modeled independently with its corresponding frequency as in~\eqref{eq:oscillation}, and then combined to produce the total response of the aquifer.

Since the solution is linear in time, we assume the solution (after some initial time has passed) can be represented as
\begin{equation} \label{eqn:measurementequation}
 \phi(\bx,t) = \Re(\Phi(\bx) \exp(i\omega t) )
\end{equation}
where $\Re(\cdot)$ is the real part and $\Phi(\bx)$ is known as the phasor, is a function of space only and contains information about the phase and amplitude of the signal. Assuming this solution, the equations~\eqref{eqn:timedomain} in the phasor domain are, 
\begin{align} \label{eqn:phasor1}
- \nabla \cdot \left(K(\bx) \nabla \Phi(\bx)\right) + i\omega S_s(\bx)  \Phi(\bx)  = &  \quad Q_0\delta(\bx-\bx_s), & \bx \in \Omega \\ \nonumber
\Phi(\bx)  = & \quad 0, &\quad  \bx \in \partial \Omega_D \\ \nonumber
\nabla \Phi(\bx) \cdot \textbf{n}  = & \quad 0, &\qquad \bx \in \partial \Omega_N \nonumber
\end{align}
The differential equation~\eqref{eqn:phasor1} along with the boundary conditions are discretized using FEniCS~\cite{LoggMardalEtAl2012a, LoggWells2010a, LoggWellsEtAl2012a} by using standard linear finite elements. Solving it for several frequencies results in system of shifted equations of the form 
\begin{equation}
 \label{eqn:genshifted}
  \left( K + \sigma_j M \right) x_j  = b \qquad j=1,\dots,n_f 
\end{equation}
where, $K$ and $M$ are the stiffness and mass matrices, respectively, that arise precisely from the discretization of~\eqref{eqn:phasor1}.

\subsection{The Geostatistical Approach}
The Geostatistical approach (described in the following papers~\cite{kitanidis1995quasi,kitanidis2010bayesian,kitanidis2007on}) is one of the prevalent approaches for solving stochastic inverse problems. The idea is to represent the unknown field as the sum of a few deterministic low-order polynomials and a stochastic term that models small-scale variability. Inference from the measurements is obtained by invoking the Bayes' theorem, through the posterior probability density function which is the product of two parts - likelihood of the measurements and the prior distribution of the parameters. Let $s(\bx) \in \mathbb{R}^{N_s}$ be the function to be estimated, here the log conductivity, and let it be modeled by a Gaussian random field. After discretization, the field can be written as $s \sim {\cal{N}}(X\beta,Q)$. Here $X$ is a matrix of low-order polynomials, $\beta$ are a set of drift coefficients to be determined and $Q$ is a covariance matrix with entries $Q_{ij} = \kappa(\bx_i,\bx_j)$, and $\kappa(\cdot,\cdot)$ 
is 
a generalized covariance kernel~\cite{christakos1984problem}. The measurement equation can be written as, 
\begin{equation}
 y = h(s) + v, \qquad v \sim {\cal{N}}(0,R)
\end{equation}
where $y \in \mathbb{R}^{N_y}$ represents the noisy measurements and $v$ is a random vector of observation error with mean zero and covariance matrix $R$. The matrices $R$, $Q$ and $X$ are part of a modeling choice and more details to choose them can be obtained from the following references~\cite{kitanidis1995quasi}. The operator $h: \mathbb{R}^{N_s}\rightarrow \mathbb{R}^{N_y}$ is known as the parameter-to-observation map or {\it measurement operator}, with entries that are the coefficients of the oscillatory terms in the expression, 
\begin{equation} 
\label{eqn:measurement}
 \int_\Omega \Re\left\{e^{i\omega t}\Phi(\bx) \delta(\bx-\bx_i)\right\} d\bx
\end{equation}
where $\bx_i$, is the location of the measurement sensor and $i=1,\dots,n_y$, where  $n_y$ is the number of measurement locations. At each measurement location, two coefficients are measured for every frequency. In all, we have $N_y=2n_fn_y$ measurements, where $n_f$ is the number of frequencies.

Following the geostatistical method for quasi-linear inversion \cite{kitanidis1995quasi}, we compute $\hat{s}$ and $\hat\beta$ corresponding to the maximum-a-posteriori probability which is equivalent to computing the solution to a weighted nonlinear least squares problem. To solve the optimization problem, the Gauss-Newton algorithm is used. Starting with an initial estimate for the field $s_0$, the procedure is described in algorithm~\ref{alg:quasi}.

\begin{algorithm}[!ht]
\begin{algorithmic}[1]
\STATE  Compute the $N_y \times N_s$ Jacobian $J$ as, 
\begin{equation} 
 J_k = \at{\frac{\partial{h}}{\partial{s}}} {s = {s}_k} 
\end{equation}
\STATE  Solve the system of equations, 

\begin{equation} \label{eq:inversion}
\left( \begin{array}{cc} 
        J_k Q J_k^T  + R & J_kX \\
	 \left(J_k X\right)^T & 0 
       \end{array}
       \right) 
       \left( \begin{array}{c}{\xi_{k+1}}\\ {\beta_{k+1}} \end{array} \right) = 
\left( \begin{array}{c} y - h({s}_k) + J_k{s}_k \\ 0 \end{array} \right) 
\end{equation}

\STATE The update $s_{k+1}$ is computed by, 
\begin{equation} 
 s_{k+1} = X \beta_{k+1} + Q J_k^T \xi_{k+1} 
\end{equation}

\STATE Repeat steps $1-3$ until the desired tolerance has been reached. (If necessary, add a line search).
\end{algorithmic}
\caption{Quasi-linear Geostatistical Approach}
\label{alg:quasi}
\end{algorithm}

Algorithm~\ref{alg:quasi} requires, at each iteration, computation of the matrices $QJ_k^T$ and $J_kQJ_k^T$. Since the prior covariance matrix $Q$ is dense, a  straightforward computation of $QJ_k^T$ can be performed in ${\cal{O}}(N_yN_s^2)$. However, for fine grids, i.e., when the number of unknowns $N_s$ is large, storing $Q$ can be expensive in terms of memory and computing $QJ_k^T$ can be computationally expensive. For regular equispaced grids and covariance kernels that are stationary or translation invariant, an FFT based method can be used to reduce the storage costs of the covariance matrix $Q$ to ${\cal{O}}(N_s)$ and cost of matrix-vector product to ${\cal{O}}(N_s\log N_s)$. For irregular grids, the Hierarchical matrix approach can be used to reduce the storage costs and cost of approximate matrix-vector product to ${\cal{O}}(N_s\log N_s)$ for a wide variety of covariance kernels~\cite{saibaba2012efficient}. Thus, in either situation, the cost for computing $QJ_K^T$ can be done in ${\cal{O}}(N_yN_s\
\log N_s)$ and the cost of computing $J_kQJ_k^T$ is  ${\cal{O}}(N_sN_y\log N_s + N_sN_y)$.

\subsection{Sensitivity Matrix computation}\label{sec:sensitivity}
Computing the Jacobian matrix $J_k$ at each iteration is often an expensive step. Although explicit analytical expressions for the entries are nearly impossible, several approaches exist. One simple approach is to use finite differences, but this approach is expensive because it requires as many $N_s+1$ runs of the forward problem, i.e. one more than the number of parameters to be estimated. For large problems and on finely discretized grids, the number of unknowns can be quite large and so this procedure is not feasible. 

To reduce the computational cost associated with calculating the sensitivity matrix we use the adjoint state method (see for example,~\cite{sun1990coupled}). This approach is exact and is computationally advantageous when the number of measurements is far smaller than the number of unknowns. For a complete derivation of the adjoint state equations for oscillatory hydraulic tomography, refer to~\cite{cardiff2012multi}. For the type of measurements described in~\eqref{eqn:measurement}, the entries of the sensitivity matrix can calculated by the following expression for $j=1,\dots,N_s$ 
\begin{equation}
\label{eqn:sensitivity}
 \frac{\partial h}{\partial s_j} = \int_\Omega \Re\left\{ e^{i\omega t}\left( \left[i \omega \frac{\partial S_s(\bx)}{\partial s_j} \Phi - \frac{\partial Q_0}{\partial s_j} \right] \Psi_{\omega} + \frac{\partial K(\bx)}{\partial s_j} \nabla \Phi \cdot \nabla \Psi_{\omega}\right)  \right\}d\bx
\end{equation}
Since at measurement location corresponding to each frequency, two measurements are obtained from the coefficients of the oscillatory terms, so the Jacobian matrix has $N_y\times N_s$ entries where, $N_y = 2n_fn_y$. Here, $\Psi_{\omega}$ is the known as the {\it adjoint solution} that depends on the measurement location $\bx_m$ and the forcing frequency $\omega$. It satisfies the following system of equations 
\begin{align}\label{eqn:adjoint} 
 - \nabla \cdot \left(K \nabla \Psi_{\omega } \right) +  i \omega S_s\Psi_{ \omega}  = & \quad- \delta(\bx-\bx_m), &\quad \bx \in \Omega \\ \nonumber 
\Psi_{\omega}  = & \quad 0, &\quad \bx \in \partial \Omega_D  \\ \nonumber
\nabla \Psi_{\omega}(\bx) \cdot \textbf{n}  = & \quad 0, &\quad \bx \in \partial \Omega_N \nonumber
\end{align}
where, $\bx_m$ is the measurement location and $\omega$ is the particular frequency. The procedure for calculating the sensitivity matrix can thus be summarized as follows. 
\begin{algorithm}[!ht]
\begin{algorithmic}
 \STATE 1. For a given field ${s}$, solve the forward problem for $\Phi$. 
 \STATE 2. For each measurement and frequency $\omega$, solve the adjoint problem for $\Psi_\omega$. 
 \STATE 3. Compute the integral in~\eqref{eqn:sensitivity} to calculate the sensitivity. 
\end{algorithmic}
\caption{Computing Sensitivity Matrix}
\label{alg:sensitivity}
\end{algorithm}
Since~\eqref{eqn:sensitivity} is evaluated for all $s_j$ for each measurement, the adjoint state method requires only $N_y+1$ forward model solves to compute the sensitivity matrix. Thus, when the number of measurements is far fewer than the number of unknowns, the adjoint state method provides a much cheaper alternative for computing the entries of the Jacobian matrix. This is typically the case in hydraulic tomography, where having several measurement locations is infeasible because it requires digging new wells.

Further, we realize that equation~\eqref{eqn:sensitivity} takes the same form as equation~\eqref{eqn:genshifted} for multiple frequencies. Thus, we can use the algorithms developed in section~\ref{sec:krylov} to solve the system of equations~\eqref{eqn:sensitivity} for as many right hand sides as measurements. It is possible to devise algorithms for multiple right hand sides in the context of shifted systems~\cite{meerbergen2010lanczos,darnell2008deflated} but we will not adopt this approach.

\section{Numerical Experiments and Results} \label{sec:numerical}

We present numerical results for the Krylov subspace solvers and its application to OHT. As mentioned before, we use the FEniCS software~\cite{LoggMardalEtAl2012a, LoggWells2010a, LoggWellsEtAl2012a} to discretize the appropriate partial differential equations. We use the Python interface to FEniCS, with uBLASSparse as the linear algebra back-end. For the direct solvers we use SuperLU~\cite{superlu99} package that is interfaced by Scipy whereas for the iterative solver we use an algebraic multigrid package PyAMG~\cite{BeOlSc2011}, with smoothed aggregation along with BiCGSTAB iterative solver. In the following sections, for brevity, we only most results for FOM solver but we observed similar results for the GMRES method as well.  This is also suggested by the result in proposition~\ref{prop:gmres}.	

\subsection{Krylov subspace solver}\label{sec:solverexperiments}
In this section, we present some of the results of the algorithms that we have described in section~\ref{sec:krylov}. We now describe the test problem that we shall use for the rest of the section. We consider a $2$D aquifer in a rectangular domain with Dirichlet boundary conditions on the boundaries. For the log-conductivity field $\log K(\bx)$, we consider a random field generated using an exponential covariance kernel $\kappa(\bx,\textbf{y}) = 4\exp(-2\normtwo{\bx-\textbf{y}}/L)$ using the algorithm described in \cite{dietrich1993fast}. Other parameters used for the model problem are summarized in table~\ref{tab:parameters}. We choose $200$ frequencies evenly spaced between the minimum and maximum frequencies, which results in $200$ systems each of size $90601$.

\begin{table}[h] 
 \centering
\begin{tabular}{|l|l|l|}
\hline
Definition &  Parameters  & Values \\
 \hline 
Aquifer length&  L (m) & 500 \\
Specific storage &  $\log S_s$ (m$^{-1}$) &  $-11.52$\\
 Mean conductivity  &  $\mu(\log K)$ (m/s) & $ -11.02$ \\
Variance of conductivity   & $\sigma^2(\log K) $ & $1.42$  \\
 Frequency range &$\omega$ ($s^{-1}$) &  $[\frac{2\pi}{600},\frac{2\pi}{3}]$ \\ 
\hline
\end{tabular}
\caption{Parameters Chosen For Test Problem}\label{tab:parameters}
\end{table}

First, we motivate the need for multiple preconditioners to solve the shifted system of equations~\eqref{eqn:genshifted}. We begin by looking at the number of iterations taken by restarted GMRES without a preconditioner and using a single preconditioner. We choose a preconditioner of the form $K+\tau M$, for five different values of $\tau$. For illustration purposes, we use a direct solver to invert the preconditioned systems. These values represent the minimum frequency, the average frequency and the maximum frequency in the parameter range. The number of iterations corresponding to restarted GMRES ($30$) for each of the system is computed and displayed in figure~\ref{fig:itercountsingleprecond}. We observe that the number of iterations corresponding to the systems increases as the frequency of the system decreases. When we use a preconditioner $K + \tau M$, the systems with frequencies nearby $|\tau|$ converge rapidly. However, systems with frequencies further from $|\tau|$ converge more slowly, the 
further away they are from the frequency of the preconditioned system $|\tau|$. This is consistent with the analysis in section~\ref{sec:geneigen} and in particular, proposition~\ref{prop:fom}. Thus, no single preconditioner effectively preconditions all the systems in the given frequency range. Not surprisingly, choosing $|\tau|$ in the center of the frequency range seems to be the best choice. Thus, in order to make the iterative method competitive, we consider using multiple preconditioners to solve the shifted system~\eqref{eqn:genshifted}.

\begin{figure}[!htbp]
\centering
\includegraphics[scale=0.29]{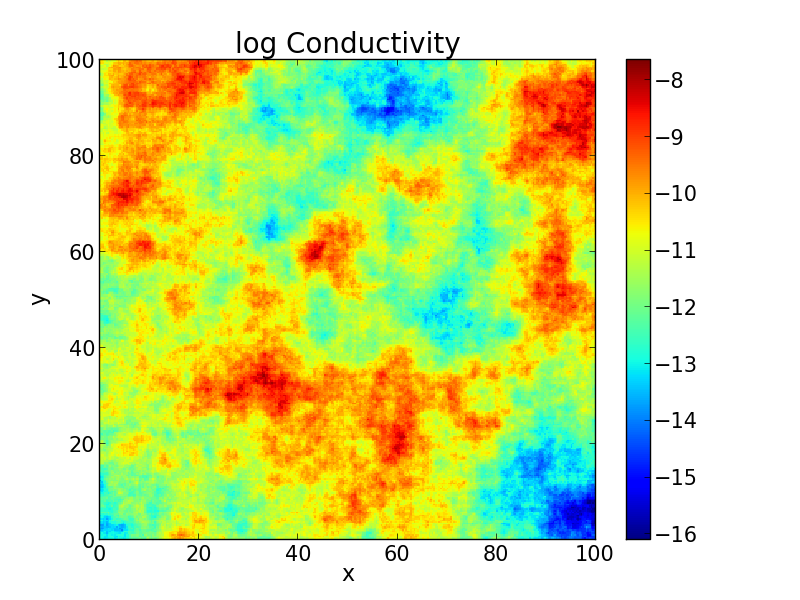}
\includegraphics[scale=0.29]{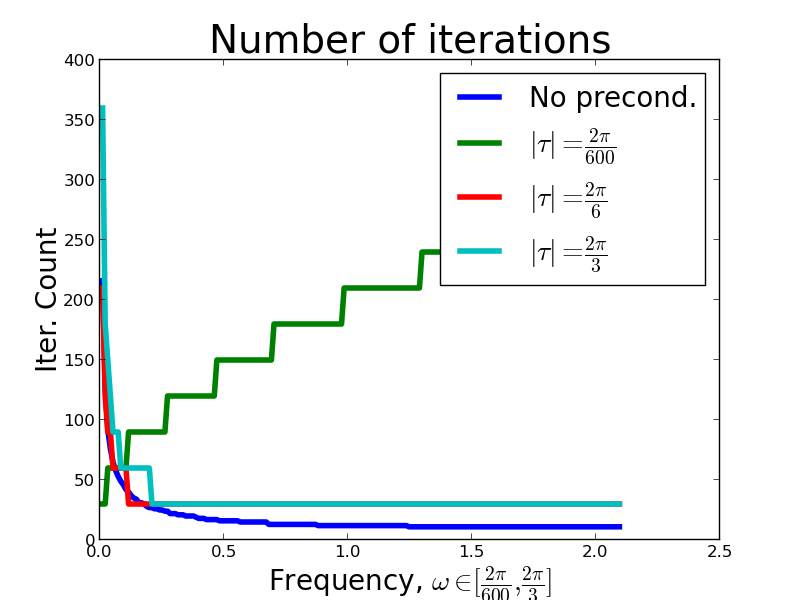}
\caption{(left) The log conductivity field that we use for the test problem with $90601$ grid points, and (right) iteration count for restarted GMRES ($30$) for unpreconditioned case and also with single preconditioners of the form $K + \tau M$, with $|\tau| \in \{ \frac{2\pi}{600},\frac{2\pi}{6},\frac{2\pi}{3} \}$. Results indicate that for the particular choices made, $|\tau| = \frac{2\pi}{6}$ which is roughly at the center of the frequency range, performs best.}
\label{fig:itercountsingleprecond}
\end{figure}

We choose preconditioners of the form $K + \tau_k M , k = 1,\dots,m$, where $m$ is the maximum dimension of the Arnoldi iteration. From figure~\ref{fig:itercountsingleprecond}, it is clear that the systems with smaller frequencies converge slower, so we choose the values of $|\tau_k|$ that are distributed closer to the origin. In particular, we choose the values of $\tau_k$ such that they are evenly spaced on a log scale in the domain $\omega \in [\frac{2\pi}{600},\frac{2\pi}{3}]$. For example, for $n_p = 5$, the distribution of $|\tau|$ is illustrated in figure~\ref{fig:multipletau}. Now, let the possible values that $\tau$ can take be labeled as $\bar{\tau} = \{ \bar{\tau}_1,\dots,\bar{\tau}_{n_p} \}$  where, $n_p$ is the number of distinct number of preconditioner frequencies. Then, the first $m_1$ values of $\tau_k$ are assigned $\bar{\tau}_1$, the next $m_2$ values of $\tau_k$ are assigned $\bar{\tau}_2$ and so on. We also have $m = \sum_{k=1}^{n_p}m_k$. We pick $m_k = m/n_p$. If the algorithm has not 
converged in $m$ iterations, we restart using the method in section~\ref{sec:restarting} if we are using a direct solver as the preconditioner. Else, we recycle the same sequence of preconditioners. We implemented both algorithms to invert the preconditioner matrices - using direct solver and using an iterative solver which is an algebraic multigrid preconditioned BiCGSTAB. Using $n_p = 5$ and $m=40$ along with the scheme to choose the preconditioner frequencies described above, we observed that the number of iterations (and hence, matrix-vector products) in both cases were less than $40$.  

\begin{figure}
\centering
\includegraphics[scale = 0.5]{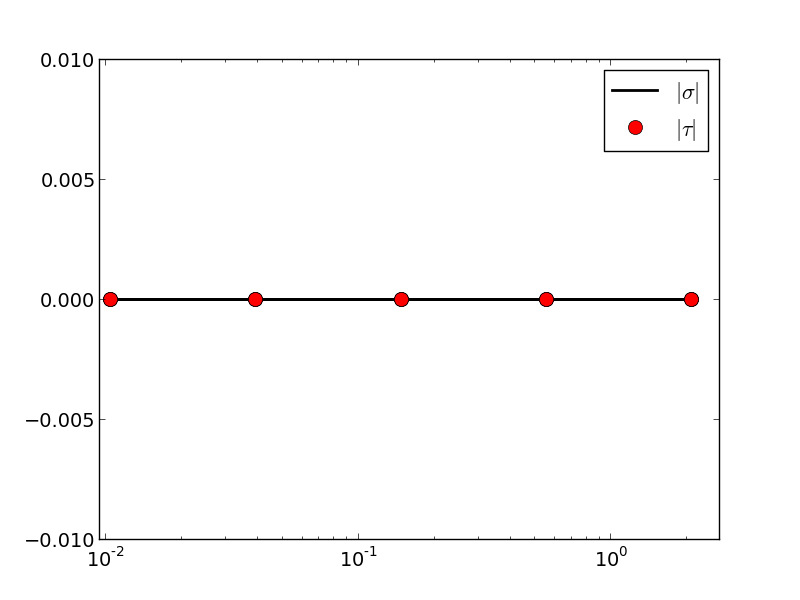}
\caption{Choice of frequencies $|\tau|$ for the preconditioners $ K + \tau_k M$. Here, for illustration, we choose $n_p = 5$ and the distinct values of the preconditioned frequencies are evenly spaced on a log scale in the domain $\omega \in [\frac{2\pi}{600},\frac{2\pi}{3}]$. }
\label{fig:multipletau}
\end{figure}

Finally, we present the comparison in terms of the run time of our algorithm compared to solving each system using a direct solver. The results are presented in figure~\ref{fig:timing}. In the plots, ``Direct'' implies that every system is solved individually by a direct solver. ``Flexible'' algorithm uses $5$ different preconditioners (see figure~\ref{fig:multipletau} with a direct solver for inverting the preconditioners, and solves the FOM subproblem, whereas ``Inexact'' uses the preconditioners and inverts the preconditioners using an iterative solver. We see that both the ``Flexible'' and ``Inexact'' algorithms outperform the ``Direct'' approach even for a small number of frequencies. A relative tolerance of $\normtwo{r_m(\sigma_j)} / \normtwo{r_0(\sigma_j)} \leq 10^{-10}$ was used as stopping criteria for all systems $j=1,\dots,n_f$ and all systems converged within $40$ iterations. Although, the ``Inexact'' algorithm seems to behave nearly independent of number of frequencies, its runtime is longer 
than the ``Direct'' approach. This is due to the fact that the PyAMG solver requires an additional $17$ matvecs on average per inner iteration, totaling $706$ matvecs with $K+\tau_kM$ with $k=1,\dots,m$ and $m=40$.

\begin{figure}
\centering
\includegraphics[scale=0.5]{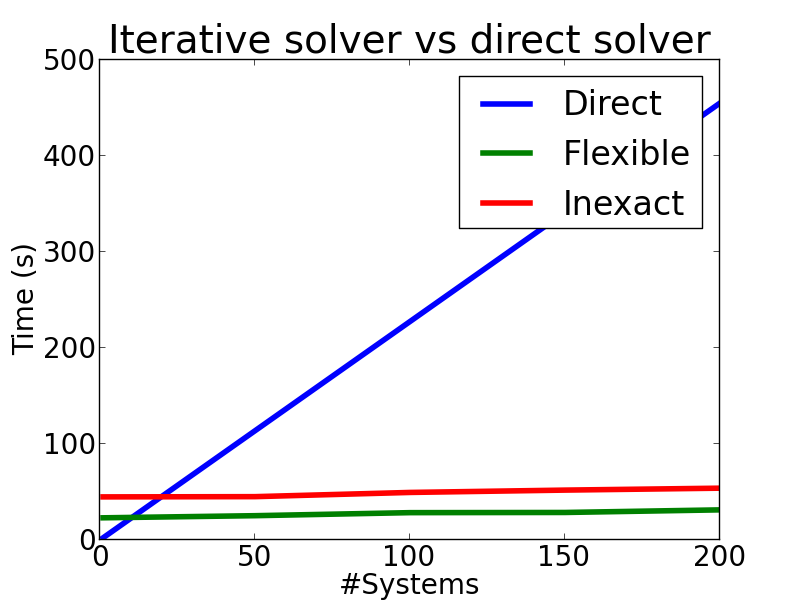}
\caption{Comparison of time for the different algorithms. ``Direct'' implies that every system is solved individually by a direct solver. ``Flexible'' algorithm uses $5$ different preconditioners (see figure~\ref{fig:multipletau} with a direct solver for inverting the preconditioners, and solves the FOM subproblem, whereas ``Inexact'' uses the preconditioners and inverts the preconditioners using an iterative solver. We see that both the ``Flexible'' and ``Inexact'' algorithms outperform the ``Direct'' approach even for a small number of frequencies. The system size is $90601$. }
\label{fig:timing}
\end{figure}

\begin{figure}
\centering
\includegraphics[scale = 0.5]{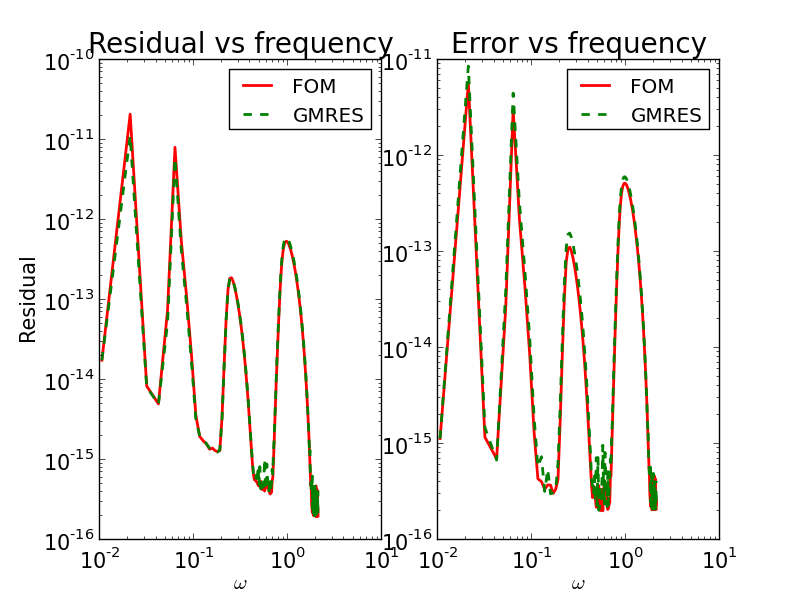}
\caption{(left) residual vs frequencies for FOM and GMRES and (right) error vs frequency for FOM and GMRES. All systems converged within 40 iterations. A tolerance of $\normtwo{r_m(\sigma_j)} / \normtwo{r_0(\sigma_j)} \leq 10^{-10}$ was used for all systems $j=1,\dots,n_f$. The preconditioner matrices were inverted using a direct solver. }
\label{fig:reserrordirect}
\end{figure}

\begin{figure}
\centering
\includegraphics[scale = 0.5]{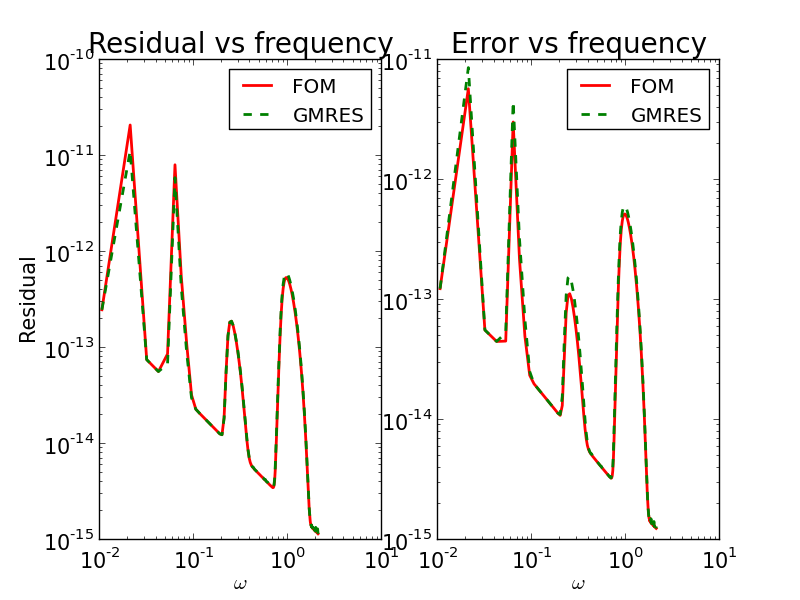}
\caption{(left) residual vs frequencies for FOM and GMRES and (right) error vs frequency for FOM and GMRES. All systems converged within 40 iterations. A tolerance of $\normtwo{r_m(\sigma_j)} / \normtwo{r_0(\sigma_j)} \leq 10^{-10}$ was used as stopping criteria for all systems $j=1,\dots,n_f$. The preconditioner matrices were inverted using a iterative solver with $\varepsilon = 10^{-12}$ as an inner stopping criterion. }
\label{fig:reserrorflexible}
\end{figure}

In figures~\ref{fig:reserrordirect} and~\ref{fig:reserrorflexible}, we plot the residuals and error as a function of the frequency of the system. A direct solver was used for figure~\ref{fig:reserrordirect}, whereas an iterative solver (algebraic multigrid preconditioned BiCGSTAB) was used in figure~\ref{fig:reserrorflexible} with a stopping criterion $\varepsilon = 10^{-12}$ (refer to section~\ref{sec:inexact}. A relative tolerance of $\normtwo{r_m(\sigma_j)} / \normtwo{r_0(\sigma_j)} \leq 10^{-10}$ was used for all systems $j=1,\dots,n_f$ and all systems converged within 40 iterations. The behavior of residual and the error is quite similar but this is to be expected.

As discussed in section~\ref{sec:inexact}, when an inexact preconditioner is used, the flexible Arnoldi relation is no longer exact and the true residual $r_m(\sigma_j)$ and the inexact residual $\tilde{r}_m(\sigma_j)$ are no longer exactly equal. In fact, the error between them can be bounded by the relation~\eqref{eqn:resdiff}. In figure \ref{fig:reserrorflexible}, we compare the difference between the true residual and the inexact residuals with the predicted bound $\varepsilon\norm{y_m(\sigma_j)}{1}$. We see in figure~\ref{fig:resdiff} that the bound is fairly accurate. The stopping tolerances were chosen to be  $\varepsilon = 10^{-9},10^{-11}, 10^{-12}$. In fact, for larger stopping tolerances for the inner solver, the outer solver did not converge.

\begin{figure}
\centering
\includegraphics[scale = 0.5]{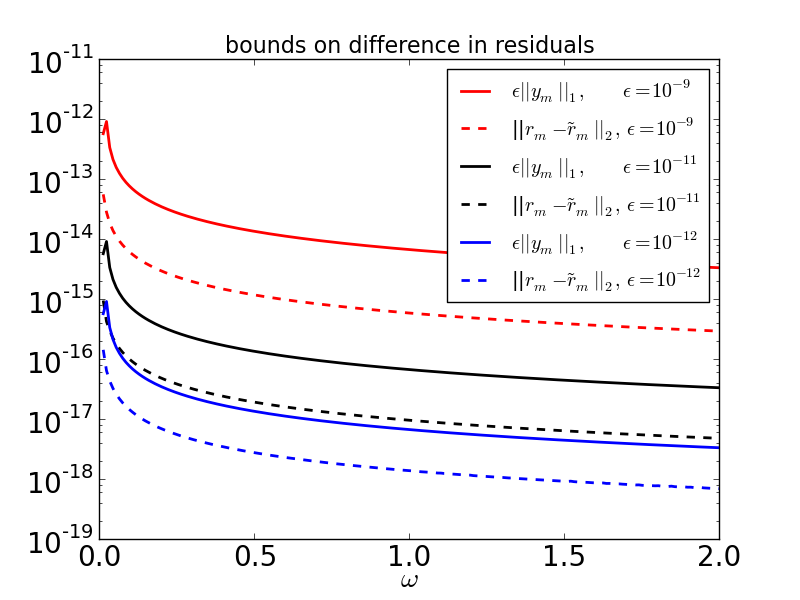}
\caption{Difference in the norm between the true and the inexact residuals, when an iterative solver is used as the preconditioner. Three different stopping tolerances were considered $\varepsilon = 10^{-9},10^{-11}, 10^{-12}$. }
\label{fig:resdiff}
\end{figure}

\subsection{Application: Tomographic reconstruction}
The objective is now to determine an known conductivity field $K(\bx)$ from discrete measurements of the head $\phi$ obtained from several pumping tests performed with multiple frequencies. Since the conductivity field needs to be positive, so that the forward problem is well-posed, we consider a log-transformation $s = \log K$. The ``true" field is taken to be that in figure~\ref{fig:trueloc} which is a scaled version of Franke's function~\cite{franke1979critical}. We choose the covariance matrix $Q$ to have entries $Q_{ij} = \kappa(\bx_i,\bx_j)$, corresponding to an exponential covariance kernel 
\[\kappa(\bx,\bF{y}) = \exp\left(-\frac{4\normtwo{\bx-\bF{y}}}{L}\right)\]
where, $L$ is the length of the domain. We also choose $R = \eta^2 I$ and $X = [1,\dots,1]^T$. We did not try to optimize the choice of covariance kernels to get the best possible reconstruction. Our goal is to study the associated computational costs. The size of our problem is chosen to be $10201$ discretization points. We assume no noise in our measurements and choose $\eta =  10^{-6}$.

The measurements are obtained by taking as the true log conductivity field, the field in figure~\ref{fig:trueloc}.  Then, the phasor is calculated by solving equations~\eqref{eqn:phasor1} with the source location and measurement locations given in figure ~\ref{fig:trueloc}. The measurements are collected for each frequency and two pieces of information are recored, the coefficients of the sine and cosine terms in equation~\eqref{eqn:measurementequation}. The inverse problem is then solved using these measurements. The frequency range that is chosen is $\omega \in [2\pi/150,2\pi/30]$. We pick $n_p = 5$ evenly spaced in log-scale in this particular frequency range and set $m=40$. All systems converged in $40$ iterations.
\begin{figure}
\centering
\includegraphics[scale = 0.29]{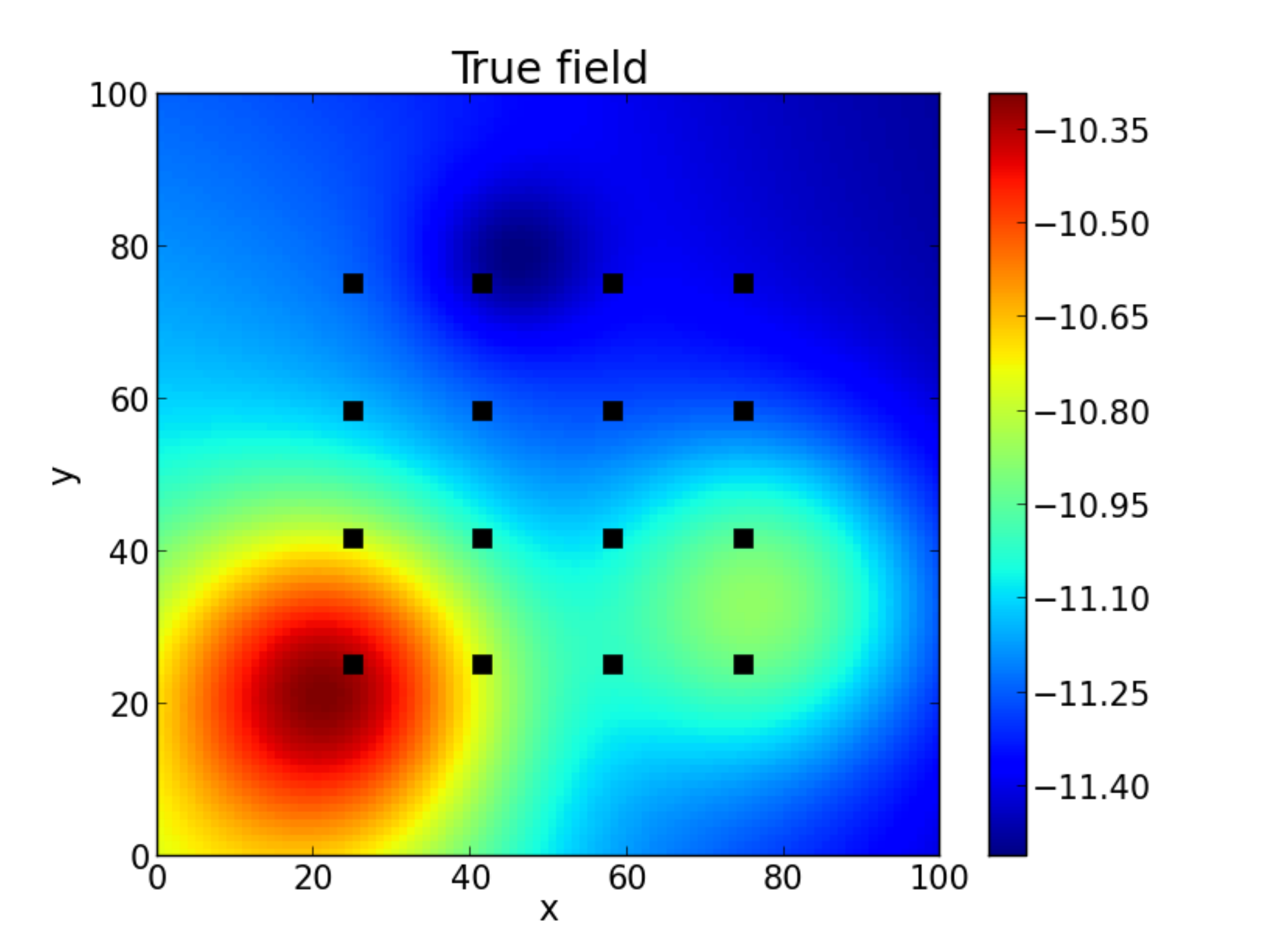}
\includegraphics[scale = 0.29]{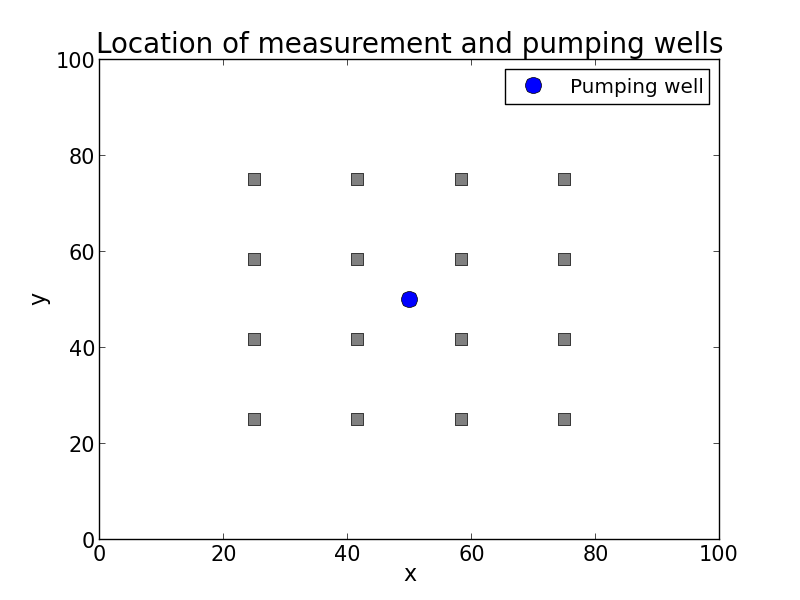}
\caption{(left) the true log conductivity field that we use for the synthetic inverse problem, and (right) location of measurement wells and the source. The size of the problem is $10201$. }
\label{fig:trueloc}
\end{figure}

The time for computing the Jacobian is listed in figure~\ref{fig:jacobian}. For the iterative solver, we use the flexible FOM solver using direct solver as preconditioner. The relative stopping tolerance we used was $10^{-10}$. Since the cost to solve systems with multiple frequencies is nearly the same as the cost to solve a single system, the time for building the Jacobian is, more or less, independent of the number of frequencies. However, when a direct solver is used to independently solve the systems for multiple frequencies, the cost for constructing the Jacobian scales linearly with the number of frequencies. This results in significant reduction in the cost for solving the inverse problem, since constructing the Jacobian is the most expensive part of solving the inverse problem. The disparity in the computation times for the Jacobian between the direct approach and the iterative procedure is exacerbated further, with larger problem sizes resulting from finer discretizations.

\begin{figure}
\centering
\includegraphics[scale = 0.5]{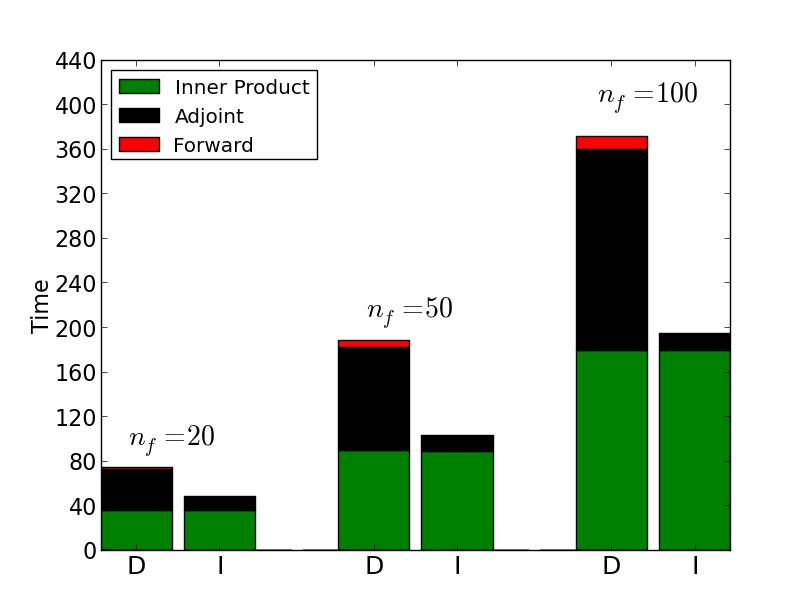}
\caption{Comparison of time taken for different components in the Jacobian. ``Forward" refers to solving the forward problem for multiple frequencies, equation~\eqref{eqn:phasor1}. ``Adjoint" refers to solving the adjoint field for multiple frequency at each measurement location, equation~\eqref{eqn:adjoint}. ``Inner Product'' refer to forming the inner product to form the rows of the Jacobian, equation~\eqref{eqn:sensitivity}. ``D'' and ``I'' refer to the direct and iterative methods respectively. The time for the forward problem using the solver is negligible, and while not visible in the plots, it is included in the construction of the bar plots.}
\label{fig:jacobian}
\end{figure}

Finally, we compare the error in the the reconstruction with multiple frequencies. Table~\ref{table:frequencies} lists the $L^2$ error in the reconstruction. We report two errors - the first being the $L^2$ error in the entire domain, the second being the $L^2$ error in the area enclosed by the measurement wells. While increasing the number of frequencies improves the error in the entire domain, as well as in the region enclosed by the measurement wells. This is the primary motivation for using multiple frequencies in the inversion. However, beyond a point, the addition of frequencies does not seem to reduce the error. This might be because there is no additional information that is obtained from the addition of measurements with these frequencies, and to further improve estimation accuracy, one would need to introduce more stimulation and observation points~\cite{cardiff2012multi}.

\begin{figure}
\centering
\includegraphics[scale = 0.29]{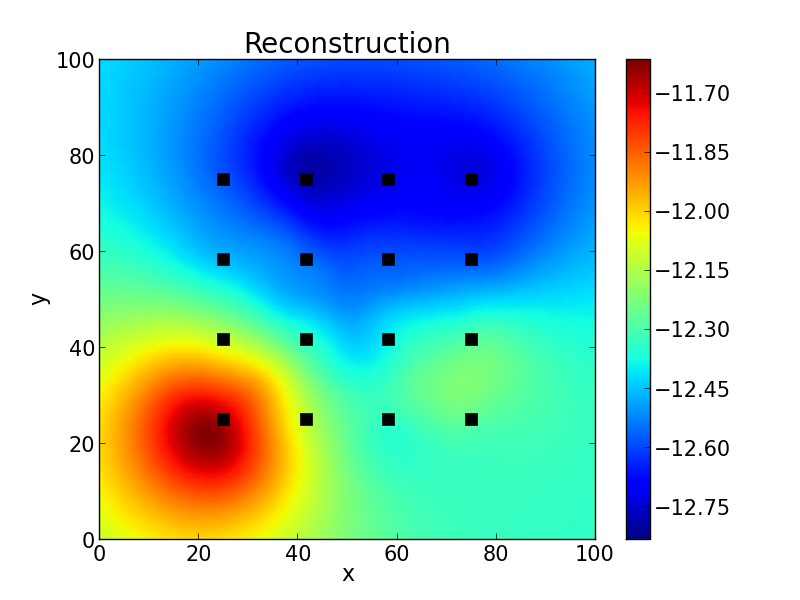}
\includegraphics[scale = 0.29]{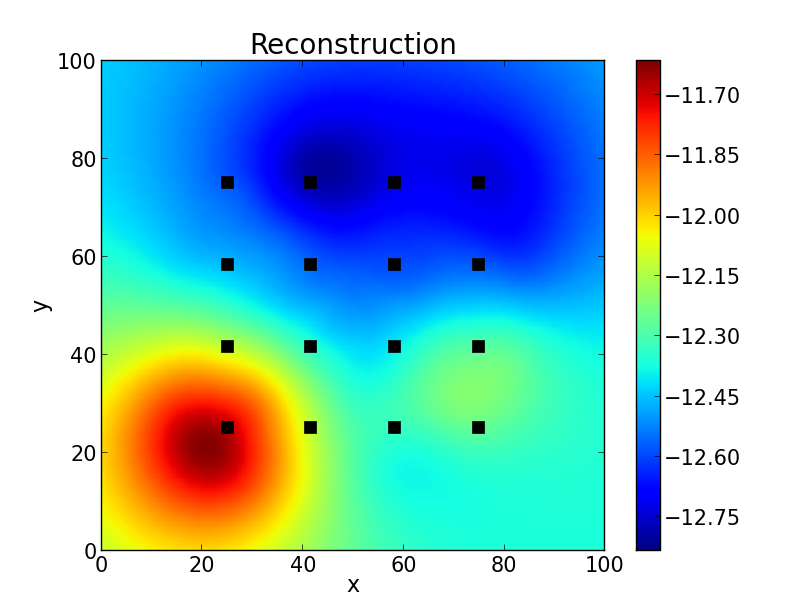}
\caption{Comparison of inversion results for log conductivity with $n_f = 1$ and $n_f = 20$ frequencies. The errors in reconstruction are reported in table~\ref{table:frequencies}.}
\label{fig:estimates}
\end{figure}

\begin{table}[h] 
 \centering
\begin{tabular}{|c|c|c|}
\hline
 $N_f$ & Total error & Error within box \\ \hline
$1$ & $0.3794$ & $0.0511$ \\ \hline
$5$ & $0.3379$ & $0.0352$ \\ \hline
$10$ & $0.3264$ & $0.0337$ \\ \hline
$20$ & $0.3180$ & $0.0328$ \\ \hline
\end{tabular}
\caption{L$^2$ error due to the reconstruction.  We report two errors - the first being the $L^2$ error in the entire domain, the second being the $L^2$ error in the area enclosed by the measurement wells. }
\label{table:frequencies}
\end{table}

\section{Conclusions}
We have presented a flexible Krylov subspace algorithm for shifted systems of the form~\eqref{eqn:genshifted} that uses multiple shifted preconditioners of the form $K+\tau M$. The values of $\tau$ are chosen in order to improve convergence of the solver for all the shifted systems. The number of preconditioners chosen varies based on the distribution of the shifts. A good rule of thumb is that the systems having shift $\sigma$ will converge faster if there a preconditioner with shift $\tau$ that is nearby $\sigma$. When the size of the linear systems is much larger, direct solvers are much more expensive. In such cases, preconditioning would be done using iterative solvers. The error analysis in section~\ref{sec:inexact} provides insight into monitor approximate residuals without constructing the true residuals. One can naturally extend the ideas in this paper to systems with multiple shifts and multiple right hand sides using either block or deflation techniques. 

We applied the flexible Krylov solver to an application problem that benefited significantly from fast solvers for shifted systems. In particular, oscillatory hydraulic tomography is a technique for aquifer characterization. However, since drilling observation wells to obtain measurements is expensive, one of the advantages of oscillatory hydraulic tomography is obtaining more informative measurements by pumping at different frequencies using the same pumping locations and measurement wells. In future studies we aim to study more realistic conditions for tomography, including a joint inversion for storage and conductivity. This would be ultimately beneficial to the practitioners. We envision that fast solvers for shifted systems would be beneficial for rapid aquifer characterization using oscillatory hydraulic tomography. 

\section{Acknowledgments}

The research in this work was funded by NSF Award 0934596, ``CMG Collaborative Research: Subsurface Imaging and Uncertainty Quantification''  and by NSF Award 1215742, `` Collaborative Research: Fundamental Research on Oscillatory Flow in Hydrogeology.'' The authors would also like to thank their collaborators Michael Cardiff and Warren Barrash for useful discussions and the two anonymous reviewers for their insightful comments.

\end{document}